\documentclass[11p,leqno]{amsart}
\usepackage{amsmath}
\usepackage{amsfonts}
\usepackage{amssymb}
\usepackage{graphicx}
\usepackage{color}
\newtheorem{theorem}{Theorem}[section]
\newtheorem*{theorem*}{Theorem}

\newtheorem{remark}[theorem]{Remark}

\numberwithin{equation}{section}

\begin{document}
\title{An alternative proof of lower bounds for the first eigenvalue on manifolds}

\author{Yuntao Zhang}
\address{School of Mathematical Sciences and Statistics, Jiangsu normal University, Xuzhou, 221116, China}
\email{yuntaozhang@xznu.edu.cn}

\author{Kui Wang}
\address{School of Mathematic Sciences, Soochow University, Suzhou, 215006, China}
\email{kuiwang@suda.edu.cn}
%

\maketitle

\begin{abstract}
Recently, Andrews and Clutterbuck \cite{AC} gave a new proof of
the optimal lower eigenvalue bound on manifolds via modulus of continuity for solutions of the heat equation.
In this short note, we give an alternative proof of Theorem 2 in \cite{AC}.
More precisely, following Ni's method (\cite[Section 6]{Ni}) we give an elliptic proof of this theorem.
\end{abstract}
\maketitle

\section{Introduction}
The aim of this short note is to give an elliptic proof of the sharp lower
bound of the first non-trivial Neumann eigenvalue on manifolds. Precisely,
let $(M, g)$ be a compact Riemannian manifold (possibly with  smooth convex boundary) with diameter $D$.
For any given constant $\kappa $,  we set
$$
c_\kappa(t)=\left\{ \begin{matrix} \cos \sqrt{\kappa} t, & \, \kappa>0,\cr
                                   1, &\, \kappa=0, \cr
                                   \cosh \sqrt{|\kappa|}t, &\, \kappa<0.\end{matrix}\right.
$$
Denote by $\lambda(M, g)$ the first non-trivial Neumann eigenvalue of $M$, which is characterised by
$$
\lambda(M, g)=\inf\left\{ \frac{\int_M |\nabla u|^2\ dx}{\int_M u^2\ dx}: u \in W^{1,2}(M)\setminus 0 \text{\quad and} \int_M u \ dx=0\right.\}.
$$
Let $\mu(n,\kappa, D)$ be the first eigenvalue of a certain one-dimensional Sturm-Liouville problem, defined by
\begin{equation}\label{1.1}
\frac{1}{c^{n-1}_{\kappa}(s)}(\Phi'(s) c^{n-1}_{\kappa}(s))'+\mu \Phi(s)=0
\end{equation}
with Neumann boundary  $\Phi'(\pm\frac{D}{2})=0$. Then eigenvalue $\lambda(M, g)$ can be bounded from below by
$\mu(n,\kappa, D)$, provided by the Ricci curvature lower bound $(n-1)\kappa$. That is
\begin{theorem}\label{thm1}
Let $(M, g)$ be a compact Riemannian manifold without boundary (or with smooth convex boundary) and denote
by $D$ its diameter. Assume further that $\operatorname{Ric}_g\ge (n-1)\kappa g$. Then
\begin{equation}\label{1.2}
\lambda(M, g)\ge \mu(n,\kappa, D).
\end{equation}
Where $\mu$ is characterised in (\ref{1.1}).
\end{theorem}
\begin{remark}
Estimate (\ref{1.2}) is sharp. For details we refer to  section 5 in \cite{An}.
\end{remark}
Theorem \ref{thm1} was proved by serval mathematicians.  Zhong and Yang \cite{ZY} proved the result for  case $\kappa=0$
and Kr\"oger proved the theorem for general case. Their proofs are based on Li--Yau type gradient estimates \cite{Li, LY}.
More Recently, Andrews and Clutterbuck also proved Theorem \ref{thm1} via the modulus of continuity in \cite{AC}, see also Andrew's survey
paper \cite{An}. In this paper, we provide a new proof, based on Ni's elliptic proof  \cite{Ni} of the eigenvalue fundamental gap.

\section{An elliptic proof of Theorem \ref{thm1}}
Consider the quotient of the oscillations of $\phi$ and $\Phi$ and let
$$
Q(x,y)=\frac{\phi(y)-\phi(x)}{\Phi(\frac{d(x,y)}{2})}
$$
on $\overline{M}\times \overline{M} \setminus \triangle$, where $\triangle:=\{(x,x):x\in \overline{M}\}$ denotes the diagonal.
The function $Q$ can be extended to a set $(\overline{M}\times \overline{M} \setminus \triangle) \bigcup UM$. Here
$UM:=\{(x, X): x\in M, |X|=1\}$.  On $UM$, with the extension $Q(x, X)$  is defined naturally as
$$
Q(x,X):=\frac{2 \langle \nabla \phi(x), X\rangle}{\Phi'(0)}.
$$
Now we divide the proof into two cases.

\textbf{Case 1:} the maximum of $Q$, which is clearly nonzero and denoted by $m$, is attained at
some $(x_0, y_0)$ with $x_0\ne y_0$. The Neumann condition  and strict convexity of $M$
forces that both $x_0$ and $y_0$ must be in $M$. Indeed if $x_0 \in \partial M$, then taking  derivative along normal direction $\nu$ at $x_0$ yields
\begin{eqnarray*}
\left.\frac{\partial}{\partial \nu}\right|_{x_0}Q(x,y_0)&=&\left.\frac{\partial}{\partial \nu}\right|_{x_0}\frac{\phi(y_0)-\phi(x)}{\Phi(\frac{d(x,y_0)}{2})}\\
&=&-\frac{m\Phi'(\frac{d(x_0,y_0)}{2})}{2\Phi(\frac{d(x_0,y_0)}{2})}\left.\frac{\partial}{\partial \nu}\right|_{x_0}d(x,y_0)\\
&<&0,
\end{eqnarray*}
which is a contradiction with the maximum assumption. Here in the last inequality, we used the convexity assumption of $M$.

Let $\gamma_0: [0,1]\rightarrow M$ be a minimizing geodesic joining $y_0$ and $x_0$
with $|\gamma'|=2s_0$.
To compute the  derivatives, we choose  'Fermi' coordinates $ \left\{e_i(s)\right\}$ $(i=1,2,\cdots, n)$ along $\gamma_0$ with
$e_n(s)=\frac{1}{2s_0}\gamma'(s)$.

Then the 1st derivative gives
\begin{equation}\label{2.1}
\nabla \phi(y_0)=\frac{m}{2}\Phi'(s_0)e_n(1),
\end{equation}
and
\begin{equation}\label{2.2}
\nabla \phi(x_0)=\frac{m}{2}\Phi'(s_0)e_n(0).
\end{equation}

Firstly, from
$$\left.\frac{d^2Q(y_0+\theta e_n(1), x_0-\theta e_{n}(0))}{d\theta^2}\right|_{\theta=0}\le 0,$$
we have
\begin{eqnarray}\label{2.3}
0&\ge& \frac{\phi_{nn}(y_0)-\phi_{nn}(x_0)}{\Phi(s_0)}+2\frac{\phi_{n}(y_0)+\phi_{n}(x_0)}{\Phi^2(s_0)}\Phi'(s_0)
-\frac{\phi(y_0)-\phi(x_0)}{\Phi^2(s_0)}\Phi''(s_0)+2m\left(\frac{\Phi'(s_0)}{\Phi(s_0)}\right)^2\nonumber\\
&=&\frac{\phi_{nn}(y_0)-\phi_{nn}(x_0)}{\Phi(s_0)}
-\frac{\phi(y_0)-\phi(x_0)}{\Phi^2(s_0)}\Phi''(s_0).\nonumber\\
\end{eqnarray}

For $i\le n-1$, we define the variation fields $V_i(s)$ along $\gamma_0(s)$ by
$$
V_i(s)=\frac{c_{\kappa}\left((2s-1)s_0\right)}{c_{\kappa}(s_0)}e_i(s).
$$
Then by the variation formulas, we have that
$$
\left.\frac{d}{dv}\right|_{v=0}|\gamma_v|=\left.\frac{1}{2s_0} g(\gamma', V_i)\right|^1_0=0,
$$
and
$$
\left.\frac{d^2}{dv^2}\right|_{v=0}|\gamma_v|=\left.\frac{1}{2s_0} g(\gamma', \nabla_{V_i} V_i)\right|^1_0+\frac{1}{2s_0}
\int_0^1|(\nabla_{\gamma'} V_i)^{\bot}|^2-\langle R(\gamma', V_i)\gamma', V_i\rangle\ ds.
$$
By the way of variation, we can require $\nabla_{V_i} V_i=0$ for $s\in [0,1]$.
Direct calculation gives
$$
\frac{1}{2s_0}\int_0^1|(\nabla_{\gamma'} V_i)^{\bot}|^2\ ds=2s_0\int_0^1(\frac{c_{\kappa}'\left((2s-1)s_0\right)}{c_{\kappa}(s_0)})^2\ ds=
\int_{-s_0}^{s_0}(\frac{c_{\kappa}'(x)}{c_{\kappa}(s_0)})^2\ dx.
$$
By the integration by parts, the definition of $c_\kappa$,  and $c_{\kappa}''+\kappa c_\kappa=0$, we obtain
$$
\int_{-s_0}^{s_0}(\frac{c_{\kappa}'(x)}{c_{\kappa}(s_0)})^2\ dx=2\frac{c_{\kappa}'(s_0)}{c_{\kappa}(s_0)}
+\int_{-s_0}^{s_0}\kappa(\frac{c_{\kappa}(x)}{c_{\kappa}(s_0)})^2\ dx.
$$
Thus
$$
\left.\frac{d^2}{dv^2}\right|_{v=0}|\gamma_v|=2\frac{c_{\kappa}'(s_0)}{c_{\kappa}(s_0)}
+\int_{-s_0}^{s_0}(\frac{c_{\kappa}(x)}{c_{\kappa}(s_0)})^2 (\kappa-\langle R(e_n, e_i)e_n, e_i\rangle)\ dx.
$$
Then from the second variation we get
$$
\phi_{ii}(y_0)-\phi_{ii}(x_0)-\Phi'(s_0)m\left( 2\frac{c_{\kappa}'(s_0)}{c_{\kappa}(s_0)}
+\int_{-s_0}^{s_0}(\frac{c_{\kappa}(x)}{c_{\kappa}(s_0)})^2 (\kappa-\langle R(e_n, e_i)e_n, e_i\rangle)\ dx\right)\le 0.
$$
Adding this inequality over $i=1,2,\cdots, n-1$ and using the curvature condition, we assert
\begin{equation}\label{2.4}
\sum_{i=1}^{n-1}\left(\phi_{ii}(y_0)-\phi_{ii}(x_0)\right)-(n-1)m\Phi'(s_0)\frac{c_{\kappa}'(s_0)}{c_{\kappa}(s_0)}\le 0.
\end{equation}

Combining (\ref{2.3}) and (\ref{2.4}) we have
$$
\triangle \phi(y_0)-\triangle \phi(x_0)-m\Phi''(s_0)-(n-1)m\Phi'(s_0)\frac{c_{\kappa}'(s_0)}{c_{\kappa}(s_0)}\le 0.
$$
That is
$$
-\lambda(M, g)\left(\phi(y_0)-\phi(x_0)\right)+m\mu(n,\kappa, D)\Phi(s_0)\le 0,
$$
which proves $\lambda(M, g)\ge \mu(n,\kappa, D)$.

\textbf{Case 2: }the maximum of $Q$ is attained at some $(x_0, X_0)\in UM$. It is easy to see that
$X_0=\frac{\nabla \phi(x_0)}{|\nabla \phi(x_0)|}$ and $m=2|\nabla \phi(x_0)|$. By the assumption, we
know $x_0\in M$.
Indeed if $x_0 \in \partial M$, then taking  derivative along normal direction $\nu$ at $x_0$ yields
\begin{eqnarray*}
\left.\frac{\partial}{\partial \nu}\right|_{x_0}|\nabla \phi(x)|^2&=&\left.\frac{\partial}{\partial \nu}\right|_{x_0}
 \left\langle\nabla \phi(x),\nabla \phi(x)  \right\rangle\\
 &=&2\left.\operatorname{Hess}\phi\left(\nu, \nabla \phi(x)\right)\right|_{x_0}\\
&=&2\left\langle\nabla_{\nabla \phi(x_0)} \nabla \phi(x_0),\nu  \right\rangle\\
 &=&-2II(\nabla \phi(x_0), \nabla \phi(x_0))\\
&<&0,
\end{eqnarray*}
contradicting with the maximum assumption. Here $II$ denotes the second fundamental form of $M$ at $x_0$.

 Now pick up an orthonormal frame $\{e_i\}$ at $x_0$ so that $e_n=X_0$.
We also parallel translate it to a
neighborhood of $x_0$.

Since $|\nabla \phi(x)|^2$ attains its maximum at the interior point $x_0$, we have that
$$
\phi_{kn}(x_0)=0$$
for any $1\le k\le n$.

Moreover, the maximum principle concludes for any $1\le k\le n-1$,
\begin{eqnarray}\label{2.5}
0\ge\phi_{kkn}\phi_n+|\phi_n|^2\langle R(e_k, e_n)e_k, e_n\rangle.
\end{eqnarray}

Let $x(s)=\exp_{x_0}(-se_n) $, $y(s) = \exp_{x_0}(se_n)$ and $g(s) = Q(x(s), y(s))$. Since Q achieves its
maximum at $(x_0,X_0)$.
We have that $g(s) \le g(0) = m$ for all $s\in(-\epsilon,\epsilon)$, which
implies that $\lim_{s\rightarrow 0}g'(s)=0$ and $\lim_{s\rightarrow 0}g''(s)\le0$.

Direct calculation shows that
\begin{eqnarray*}
g'(s)=\frac{\langle \nabla \phi(y(s)), e_n(s)\rangle+\langle \nabla \phi(x(s)),e_n(-s)\rangle}{\Phi(s)}
-g(s)\frac{\Phi'(s)}{\Phi(s)},
\end{eqnarray*}
and
\begin{eqnarray*}
g''(s)&=&\frac{\nabla^2_{e_n e_n}\phi(y(s))-\nabla^2_{e_n e_n}\phi(x(s))}{\Phi(s)}
-2\frac{\langle \nabla\phi(y(s)), e_n\rangle+\langle \nabla \phi(x(s)), e_n \rangle}{\Phi(s)}\frac{\Phi'(s)}{\Phi(s)}\\
&&-g(s)\frac{\Phi''(s)}{\Phi(s)}+2g(s)\left(\frac{\Phi'(s)}{\Phi(s)}\right)^2.
\end{eqnarray*}
Observing that $\lim_{s\rightarrow 0}\frac{g'(s)}{\Phi(s)}=g''(0)$,
and making use of the first equation above, the
second equation implies that
\begin{equation}\label{2.6}
g''(0)=2\phi_{nnn}(x_0)-2g''(0)-m\lim_{s\rightarrow 0}\frac{\Phi''(s)}{\Phi(s)}.
\end{equation}
From the equation (\ref{1.1}), it follows that
$$
\lim_{s\rightarrow 0}\frac{\Phi''(s)}{\Phi(s)}=-\lim_{s\rightarrow 0}(n-1)\frac{c_{\kappa}'(s)}{\Phi(s)}\frac{\Phi'(s)}{c_\kappa(s)}-\mu(n,\kappa, D)=(n-1)\kappa-\mu(n,\kappa, D).
$$
Then we have that
\begin{equation}\label{2.7}
2\phi_{nnn}(x_0)-m(n-1)\kappa+\mu(n,\kappa, D) m\le 0.
\end{equation}
Combining (\ref{2.5}) and (\ref{2.7}), we derive
\begin{eqnarray*}
2\langle \nabla \triangle \phi, \nabla \phi\rangle+2\text{Ric}(e_n,e_n)|\phi_n|^2
-2|\phi_n|^2(n-1)\kappa+2\mu(n,\kappa, D)|\phi_n|^2\le 0,
\end{eqnarray*}
which also implies $\lambda(M, g)\ge \mu(n,\kappa, D)$.

\section{For $p$-Laplacian}
In this section, we mainly deal with the first  non-trivial Neumann eigenvalue for $p$-Laplacian. Then
the corresponding eigenfunctions satisfy
\begin{equation}\label{3.1}
\operatorname{div}(|\nabla u|^{p-2}\nabla u)+\lambda_p |u|^{p-2}u=0.
\end{equation}

Denote by $\lambda_p(M, g)$ the first  non-trivial Neumann eigenvalue for $p$-Laplacian on $(M, g)$,
and then by Rayleigh quotient $\lambda_p(M, g)$ is characterised by
$$
\lambda_p(M, g)=\inf\left\{\frac{\int_M |\nabla u|^p\ dx}{\int_M |u|^p\ dx}: u\in W^{1,p}(M)\setminus 0 \text{\quad and} \int_M |u|^{p-2}u\ dx=0\right.\}.
$$

Similarly we denote by $\mu_p(n,\kappa, D)$ the first eigenvalue of  a certain one-dimensional Sturm-Liouville problem corresponding to $p$-Laplacian, i.e.
\begin{equation}\label{3.2}
(p-1)|\Phi'|^{p-2}\Phi''+(n-1)\frac{c_{\kappa}'}{c_{\kappa}}|\Phi'|^{p-2}\Phi' +\mu_p(n,\kappa, D) |\Phi|^{p-2}\Phi=0
\end{equation}
with $\Phi(0)=\Phi'(\pm\frac{D}{2})=0$ and $\Phi'(0)=1$. Then for p-Laplace operator, the following theorem holds true.
\begin{theorem}[see \cite{NV}]\label{thm2}
Under the same assumption as in Theorem \ref{thm1}. Then for any $1<p<\infty$ it follows that
$$
\lambda_p(M, g)\ge \mu_p(n,\kappa, D).
$$
\end{theorem}
We mention here that Valtorta \cite{Va} proved Theorem \ref{thm2} for the case $\kappa=0$ in 2012 and the main tool
used is a gradient comparison based on a generalized p-B\"ochner formula. Shortly, Naber and Valtorta \cite{NV} proved the theorem
for general $\kappa$, based on a refined gradient comparison technique and a careful analysis of the underlying model spaces.
In  survey paper \cite{An}, Andrews also proved the results for $p\le 2$ and $p>2$,  based on the modulus of continulity
 and the height-dependent gradient estimates for solutions of
the heat equation respectively.

\begin{remark}
The elliptic proof in section 2 works similarly for Theorem \ref{thm2} for case $p\le 2$.
\end{remark}
Now we are in position to use the elliptic method to show Theorem \ref{thm2} for $1<p<2$.
\begin{proof}[Proof of Theorem \ref{thm2} for $1<p<2$]
Firstly, we claim that for $1<p<2$, the maximum of $Q$ can not be attained in $UM$.
In fact if the maximum of $Q$ is attained at some  $(x_0, X_0)\in UM$, then it is easy to see that
$X_0=\frac{\nabla \phi(x_0)}{|\nabla \phi(x_0)|}$ and $m=2|\nabla \phi(x_0)|$.
From the equation (\ref{3.1}), it follows that
$$
\lim_{s\rightarrow 0}\frac{\Phi''(s)}{\Phi(s)}=-\infty,
$$
contradicting with equality (\ref{2.6}) and $\lim_{s\rightarrow 0}g''(s)\le0$. Thus we prove  the claim.

The maximum of $Q$, which is clearly nonzero and denoted by $m$, is attained at
some $(x_0, y_0)$ with $x_0\ne y_0$.

Recall from (\ref{2.1}--\ref{2.4}) that
\begin{eqnarray}\label{3.3}
0\ge(p-1)|\frac{m}{2}\Phi'(s_0)|^{p-2} \frac{\phi_{nn}(y_0)-\phi_{nn}(x_0)}{\Phi(s_0)}
-(p-1)|\frac{m}{2}\Phi'(s_0)|^{p-2} \frac{m\Phi''(s_0)}{\Phi(s_0)},\nonumber\\
\end{eqnarray}
and
\begin{equation}\label{3.4}
|\frac{m}{2}\Phi'(s_0)|^{p-2}\sum_{i=1}^{n-1}\left(\phi_{ii}(y_0)-\phi_{ii}(x_0)\right)- (n-1)m\Phi'(s_0)|\frac{m}{2}\Phi'(s_0)|^{p-2}\frac{c_{\kappa}'(s_0)}{c_{\kappa}(s_0)}\le 0.
\end{equation}

Combining (\ref{3.3}) and (\ref{3.4}) we have
$$
\triangle_p \phi(y_0)-\triangle_p \phi(x_0)-m(p-1)|\frac{m}{2}\Phi'(s_0)|^{p-2}\Phi''(s_0)- (n-1)m|\frac{m}{2}\Phi'(s_0)|^{p-2}\Phi'(s_0)\frac{c_{\kappa}'(s_0)}{c_{\kappa}(s_0)}\le 0.
$$
That is
$$
-\lambda_p(M,g)\left(|\phi(y_0)|^{p-2}\phi(y_0)-|\phi(x_0)|^{p-2}\phi(x_0)\right)+m(\frac{m}{2})^{p-2}\mu_p(n,\kappa, D)|\Phi(s_0)|^{p-2}\Phi(s_0)\le 0,
$$
which implies
\begin{eqnarray*}
m^{p-1}(\frac{1}{2})^{p-2}\mu_p (n,\kappa, D)\Phi^{p-1}(s_0)&\le&\lambda_p(M,g)\left(|\phi(y_0)|^{p-2}\phi(y_0)-|\phi(x_0)|^{p-2}\phi(x_0)\right)\\
&=&\lambda_p(M,g)\left(|\phi(y_0)|^{p-1}+|\phi(x_0)|^{p-1}\right)\\
&\le&\frac{\lambda_p(M,g)}{2^{p-2}}\left(|\phi(y_0)|+|\phi(x_0)|\right)^{p-1},
\end{eqnarray*}
where we used the condition $1<p\le 2$ in the last inequality.
Thus we conclude from the above inequality that
 $$\lambda_p(M,g)\ge \mu_p(n,\kappa, D).$$
\end{proof}


\end{document}